\documentclass[11pt]{article}
\usepackage{ifpdf}
\ifpdf
 \pdfoutput=1\pdfcompresslevel=9\pdfadjustspacing=1
 \usepackage[pdftex,bookmarks,a4paper,colorlinks]{hyperref}
 \usepackage{aeguill}
 \usepackage[pdftex]{graphicx}

 \hypersetup{bookmarksnumbered,plainpages=false,hypertexnames=false}
 \hypersetup{pagecolor=blue,linkcolor=blue,citecolor=blue,urlcolor=blue}
 \hypersetup{pdfcreator=PDFLaTeX with TPAgregTeam class}
\else
 \usepackage[latin1]{inputenc}
 \usepackage[T1]{fontenc}
 \usepackage{url}
 \usepackage[dvips]{graphicx}
\fi

\usepackage[english]{babel}
\usepackage{amsmath,amsfonts,amssymb,dsfont}
\usepackage{theorem}
\usepackage{geometry}
\usepackage{enumerate}

\newcommand{\ABS}[1]{{\left| #1 \right|}} 
\newcommand{\PAR}[1]{{\left(#1\right)}} 
\newcommand{\BRA}[1]{{\left\{#1\right\}}} 
\newcommand{\NRM}[1]{{\left\Vert #1\right\Vert}} 

\newcommand{\ind}{\mathds {1}}

\newcommand{\dE}{\mathbb{E}}

\newcommand{\dN}{\mathbb{N}}
\newcommand{\dP}{\mathbb{P}}

\newcommand{\dR}{\mathbb{R}}

\newcommand{\cB}{{\mathcal B }}

\newcommand{\cE}{{\mathcal E }}

\newcommand{\cL}{{\mathcal L }}

\newcommand{\cP}{{\mathcal P }}
\newcommand{\cX}{{\mathcal X }}

\newcommand{\N}{{\scriptscriptstyle N}}

\newtheorem{thm}{Theorem}[section]
\newtheorem{cor}[thm]{Corollary}
\newtheorem{prop}[thm]{Proposition}
\newtheorem{lem}[thm]{Lemma}
\newtheorem{defi}[thm]{Definition}
\newtheorem{rem}[thm]{Remark}

\newenvironment{proof}
               {\noindent {\textbf{Proof.}}}
               {\proofend}

\newcommand{\proofend}{\hfill $\Box{~}$}
\title{Quantitative estimates for the long time behavior of an 
ergodic variant of the telegraph process}

\author{%
  Joaquin~\textsc{Fontbona}, %
  H\'el\`ene~\textsc{Gu\'erin}, %
  Florent~\textsc{Malrieu}}

\begin{document}

\maketitle

\tableofcontents

\begin{abstract} 

Motivated by stability questions on piecewise deterministic Markov models 
of bacterial chemotaxis, we study the long time behavior of a   variant of  
the classic telegraph process having a non-constant jump rate that induces 
a  drift  towards the origin. We compute its invariant law and show 
exponential ergodicity, obtaining a quantitative control of the total 
variation distance  to equilibrium at each instant of time. These results rely on an 
exact description of the excursions of the process away from the origin and 
on  the explicit  construction of  an original coalescent coupling for both 
velocity and position. Sharpness of the obtained convergence rate is 
discussed.

\end{abstract}


\section{Introduction}
\subsection{The model and main results}

Piecewise Deterministic  Markov Process (PDMP) have been 
extensively studied in the last two decades  and received renewed 
attention in recent years in different applied probabilistic models (we refer 
to  \cite{Da} or \cite{Jacob} for  general  background). We consider the 
simple PDMP of kinetic type ${(Z_t)}_{t\geq 0}={((Y_t,W_t))}_{t\geq 0}$
with values in  $\dR\times\BRA{-1,1}$ and  infinitesimal generator 
\begin{equation}\label{eq:geninfsimple}
Lf(y,w)=w\partial_yf(y,w)+\PAR{a+(b-a)\ind_\BRA{yw>0}}(f(y,-w)-f(y,w)),
\end{equation}
where $b\geq a>0$ are given real numbers. 
That is, the  continuous component  $Y_t$ evolves according to 
$\frac{d Y_t}{dt}= W_t$
 and represents the position of a particle on the real line,  
whereas  the component $W_t$ represents the velocity of the particle 
and jumps between $+1$ and $-1$, with instantaneous state-dependent 
rate. More precisely, as long as $Y_t$ is positive the jump rate of the 
velocity is equal to $b$ if $W=+1$, and it is equal to $a$ if $W=-1$; 
the situation is reversed if $Y_t$ is negative. The case $a=b$ 
corresponds to the classical \emph{telegraph process} 
in $\dR\times \BRA{-1,+1}$ introduced by Kac \cite{kac}, in which 
case the density of $(Y_t)$ solves the damped wave equation 
$$
\frac{\partial^2p}{\partial t^2}-\frac{\partial^2p}{\partial x^2}
+a \frac{\partial p}{\partial t}=0
$$ 
called the \emph{telegraph equation}. 
The telegraph process,  as well as  its variants and its connections 
with the so-called persistent random walks have received considerable 
attention both in the physical and mathematical literature (see e.g. \cite{herrmann} 
for historical references and for some recent probabilistic  developments). 
It is well known that $(Y_t)_{t\geq 0}$ converges when $a=b$ to the 
standard one dimensional Brownian motion in the suitable scaling limit. 
Figure \ref{fi:traj} shows a path of $Y$ driven by \eqref{eq:geninfsimple} with
$a=1$ and $b=2$. 
\begin{figure}
\begin{center}
 \includegraphics[scale=0.6]{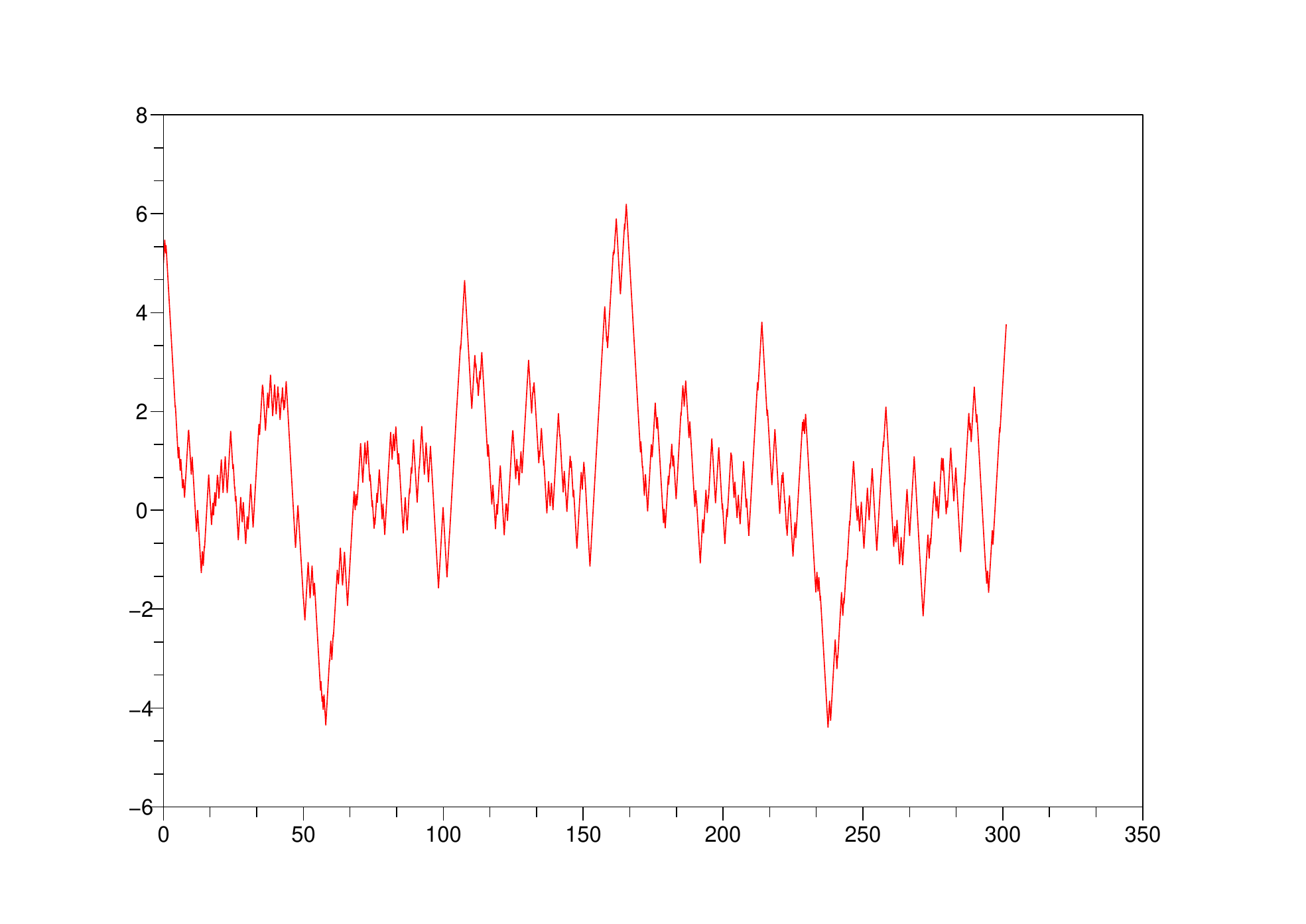}
 \caption{Trajectory of the continuous part $Y$ from the dynamics
  \eqref{eq:geninfsimple} with $a=1$ and $b=2$.}
 \label{fi:traj}
\end{center}
\end{figure}  

In this paper, we are interested in  the long-time  stability  properties 
of the process \eqref{eq:geninfsimple}  when  $b>a$. One of our   
motivations is a better understanding of  dissipation mechanisms in the setting of hyperbolic equations, where the telegraph process appears as the prototypical associated Markov  process.   A second motivation is  to make a first step in tackling questions on the trend to equilibrium 
of {\it velocity jump processes} introduced  in  \cite{othmer}, \cite{EO}, which model  the interplay between intra-cellular chemoattractant response mechanisms and collective (macroscopic) behavior of unicellular organisms.  These PDMP describe the motion of  flagellated bacteria as a sequence of  linear ``runs'', the directions of which  randomly change at rates that evolve according some simple dynamics that represent internal adaptive or excitative  responses to chemical  changes in the environment (we refer the reader to \cite{rousset}  for a deeper probabilistic description). The emergence of macroscopical drift is  expected when the response  mechanism favors longer runs in specific directions, and  has been numerically confirmed  in  \cite{othmer}, \cite{EO}. In \cite{rousset}, with the aim of developing variance reduction techniques for the numerical simulation of these models, a so-called {\it gradient sensing process} was derived  from them, in asymptotics  where 
 the response mechanisms of bacteria, roughly speaking, act infinitely fast (see Lemma 2.5  in  \cite{rousset}   for a
precise mathematical statement).  The process $(Z_t)$ above   exactly  corresponds  to the gradient sensing process  for  the particular 
   chemoattractant potential   $S(x)=c|x|\geq 0$ in $\dR$, and constitutes  a tractable toy model for the long-time behavior of the processes considered  in  \cite{othmer}, \cite{EO}, \cite{rousset} .   
    
 When $b>a$  a particle  driven 
by \eqref{eq:geninfsimple} spends in principle  more 
time moving towards the origin than away from it. Thus, a macroscopic
attraction to the origin should take place in the long run,  though in a  consistent way with the fact that the particle has constant speed.  Our  main goal is 
to clarify this picture  by determining  the invariant measure $\mu$ of 
$(Y,W)$ when $b>a$, and 
obtaining quantitative bounds ({\it i.e.} estimates that are explicit 
functions of the parameters $a$ and $b$)  for the convergence to 
$\mu$ of the law of $(Y_t,W_t)$  as $t$ goes to infinity. Denote by 
$\| \eta -\tilde \eta \|_{\mathrm{TV}}$ the total variation distance between 
two probability measures $\eta$ and $\tilde \eta$ on $\dR$ (recalled below at 
\eqref{eq:def-var-tot}). Our main result is 

\begin{thm}\label{th:no-reflection}
 The invariant probability measure $\mu$ of $(Y,W)$ is the product
 measure on $\dR\times \BRA{-1,+1}$ given by
 $$
 \mu(dy,dw)=\frac{b-a}{2}e^{-(b-a)\ABS{y}}dy 
 \otimes \frac{1}{2}(\delta_{-1}+\delta_{+1})(dw).
 $$
Moreover, denoting by $\mu_t^{y,w}$ the law of $Z_t=(Y_t,W_t)$ when issued 
from $Z_0=(y,w)$, we have, for any $y,\tilde y\in\dR$ and 
$w,\tilde w\in\BRA{-1,+1}$, 
 \begin{equation}   \label{eq:var-no-reflection}
 \NRM{\mu_t^{y,w}-\mu_t^{\tilde y,\tilde w}}_{\mathrm{TV}}%
 \leq C(a,b) e^{r(a,b)\ABS{y }\vee \ABS{ \tilde y}} e^{-\lambda_c t},
\end{equation}
where
$$
C(a,b)=\PAR{\frac{b}{a}}^{5/2}\frac{a+b}{\sqrt{ab}+b},
\quad
r(a,b)=\frac{3(b-a)}{4} \vee (b-\sqrt{ab})
\quad\text{and}\quad
\lambda_c=\frac{(\sqrt{b}-\sqrt{a})^2}{2}.
$$
\end{thm}


 We easily deduce

\begin{cor}\label{integbound}
 Let $\eta$ be a probability measure in $\dR\times\BRA{-1,+1}$ and  let 
 $\mu_t^\eta$ the law of $Z_t$ when the law of $Z_0$ is given by $\eta$. 
 Then, 
 $$
 \NRM{\mu_t^\eta -\mu}_{\mathrm{TV}}
 \leq C(a,b) \int\! e^{r(a,b)\ABS{y}}\, (\mu+\eta)(dy,dw) 
 e^{-\lambda_c t}.
 $$
\end{cor}
The upper bound \eqref{eq:var-no-reflection} is integrable 
under the invariant measure $\mu$ of the full process $(Y,W)$ 
since $r(a,b)<b-a$. Thus, Corollary \ref{integbound} is significant 
as soon as  $(y,w)\mapsto e^{r(a,b)\ABS{y}}$ is $\eta$-integrable, 
ensuring in that case  the convergence to equilibrium at exponential 
rate $\lambda_c$. Figure \ref{fi:loivrai} compares the empirical law of $Y_t$ to 
its invariant measure for successive times (the shapes might be 
compared to those presented in \cite[p. 385]{othmer}).

\begin{figure}
\begin{center}
 \includegraphics[scale=0.2]{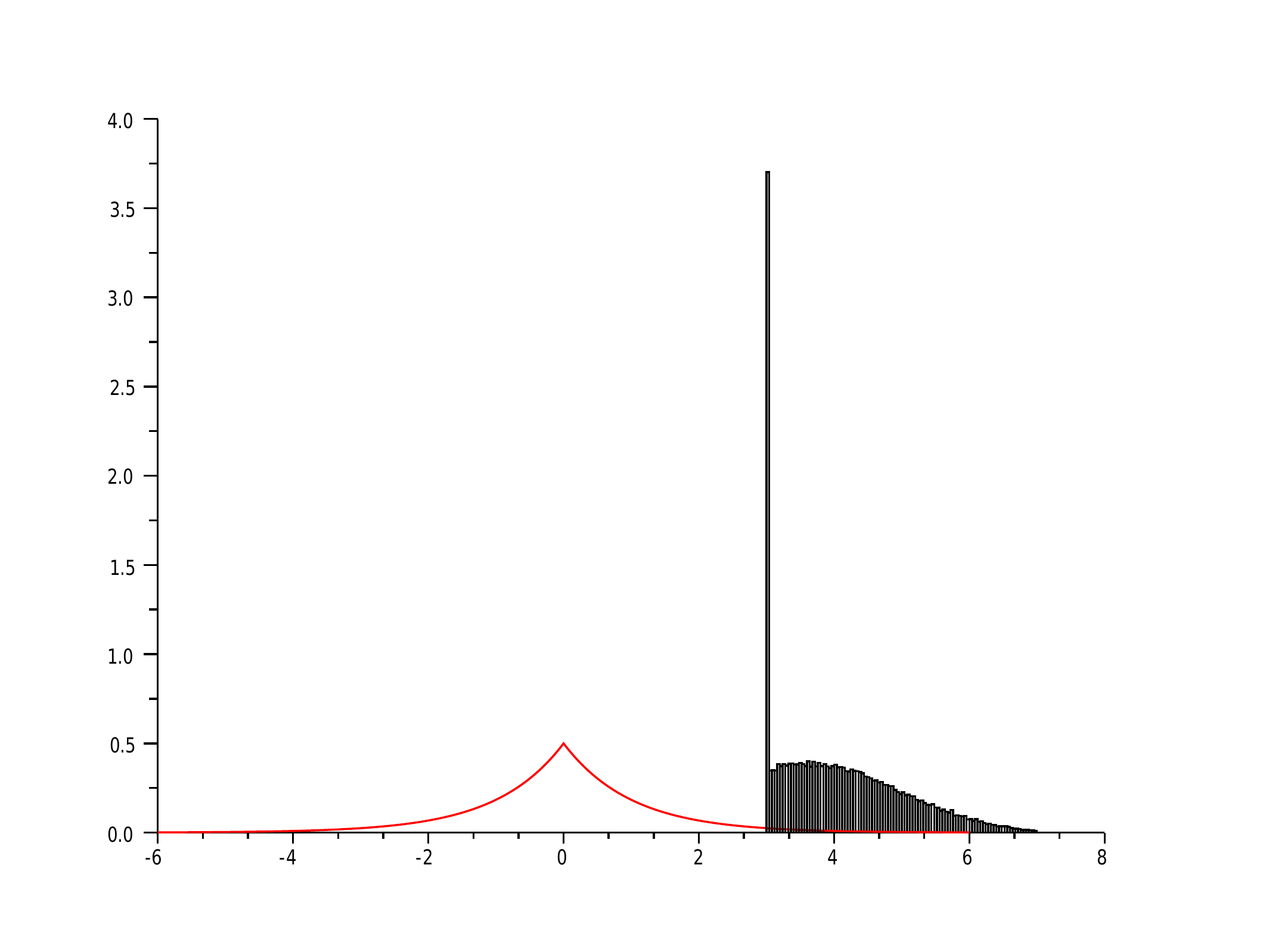}
 \includegraphics[scale=0.2]{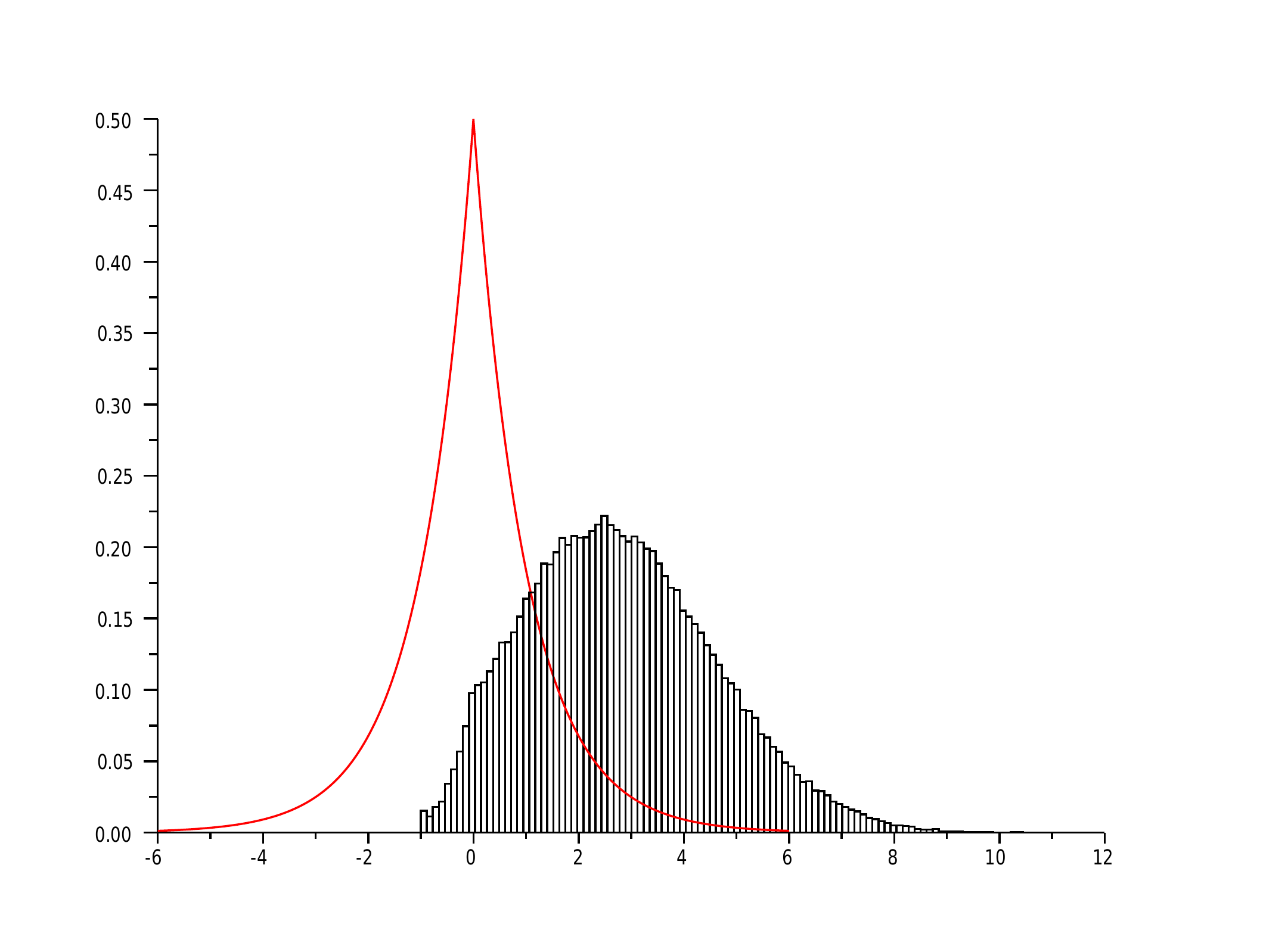}
 \includegraphics[scale=0.2]{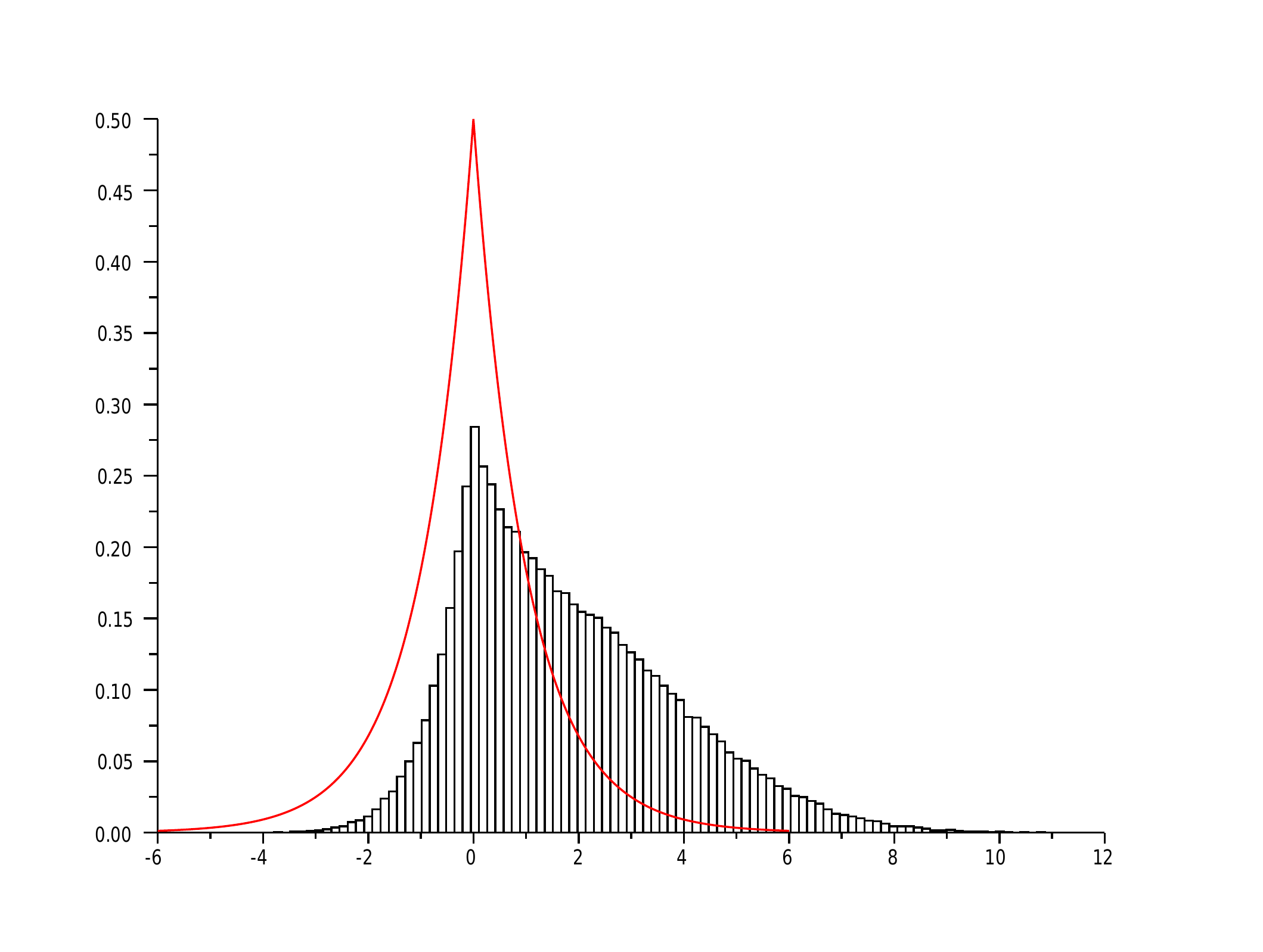}\\
 \includegraphics[scale=0.2]{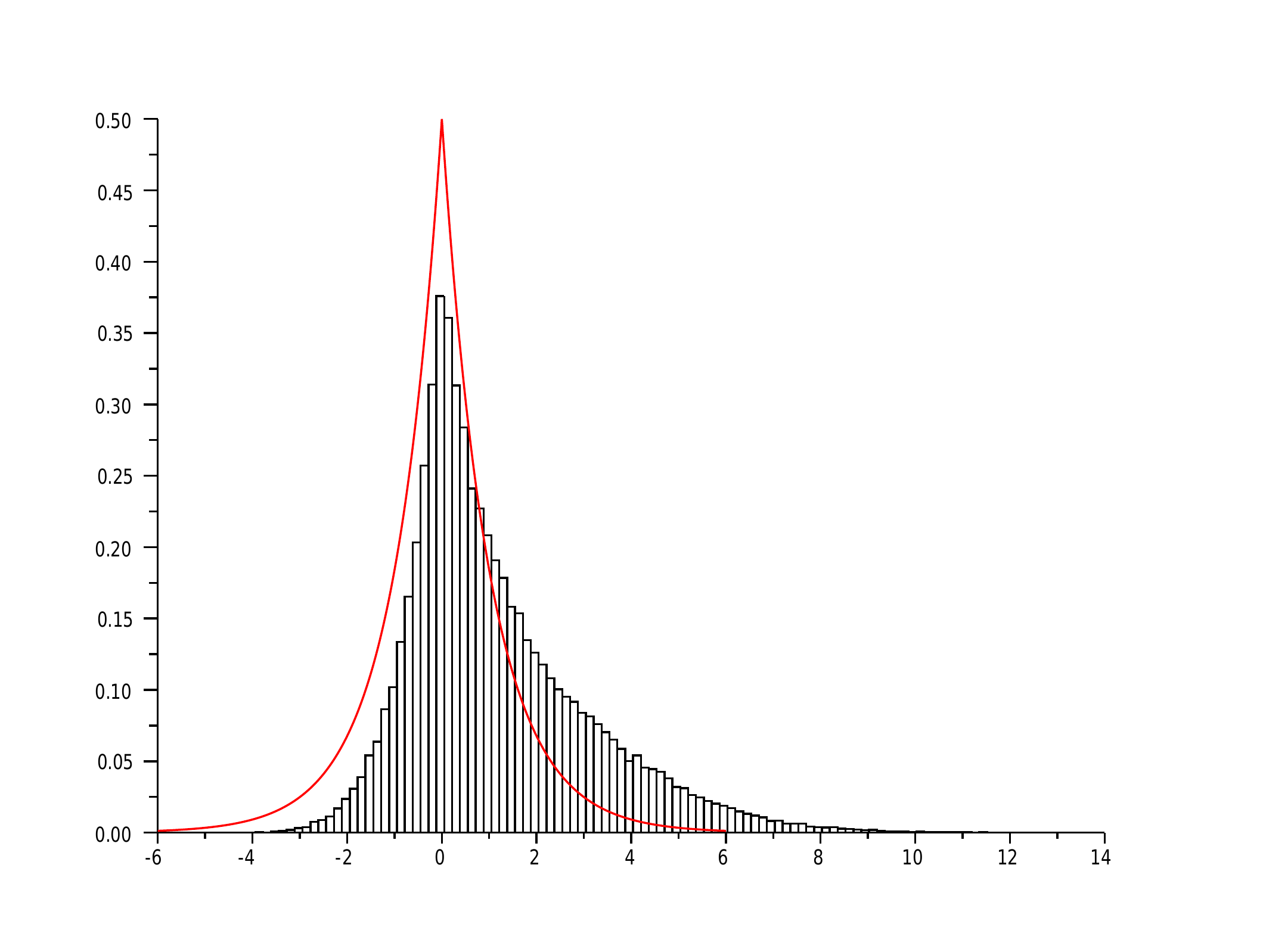}
 \includegraphics[scale=0.2]{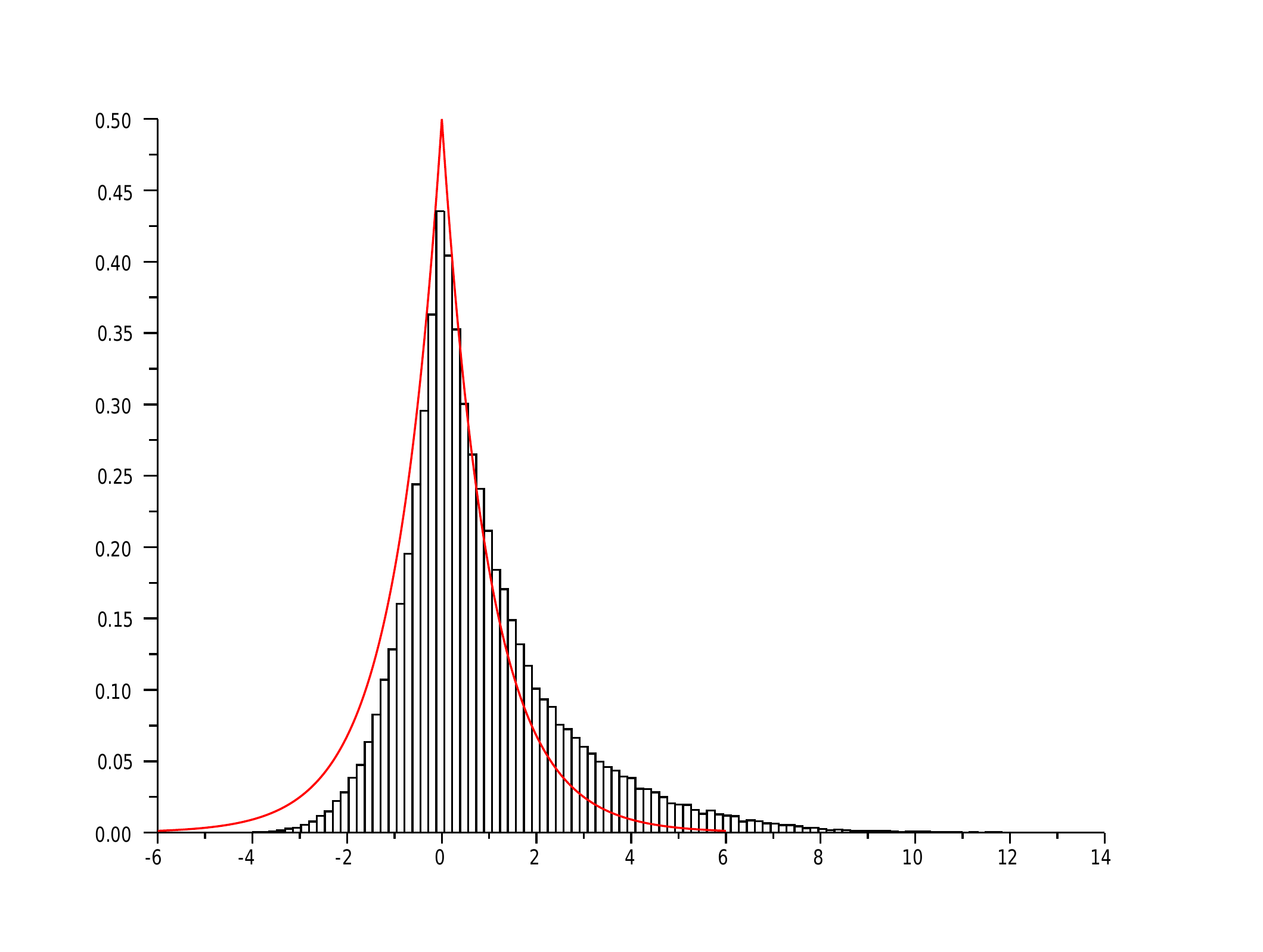}
 \includegraphics[scale=0.2]{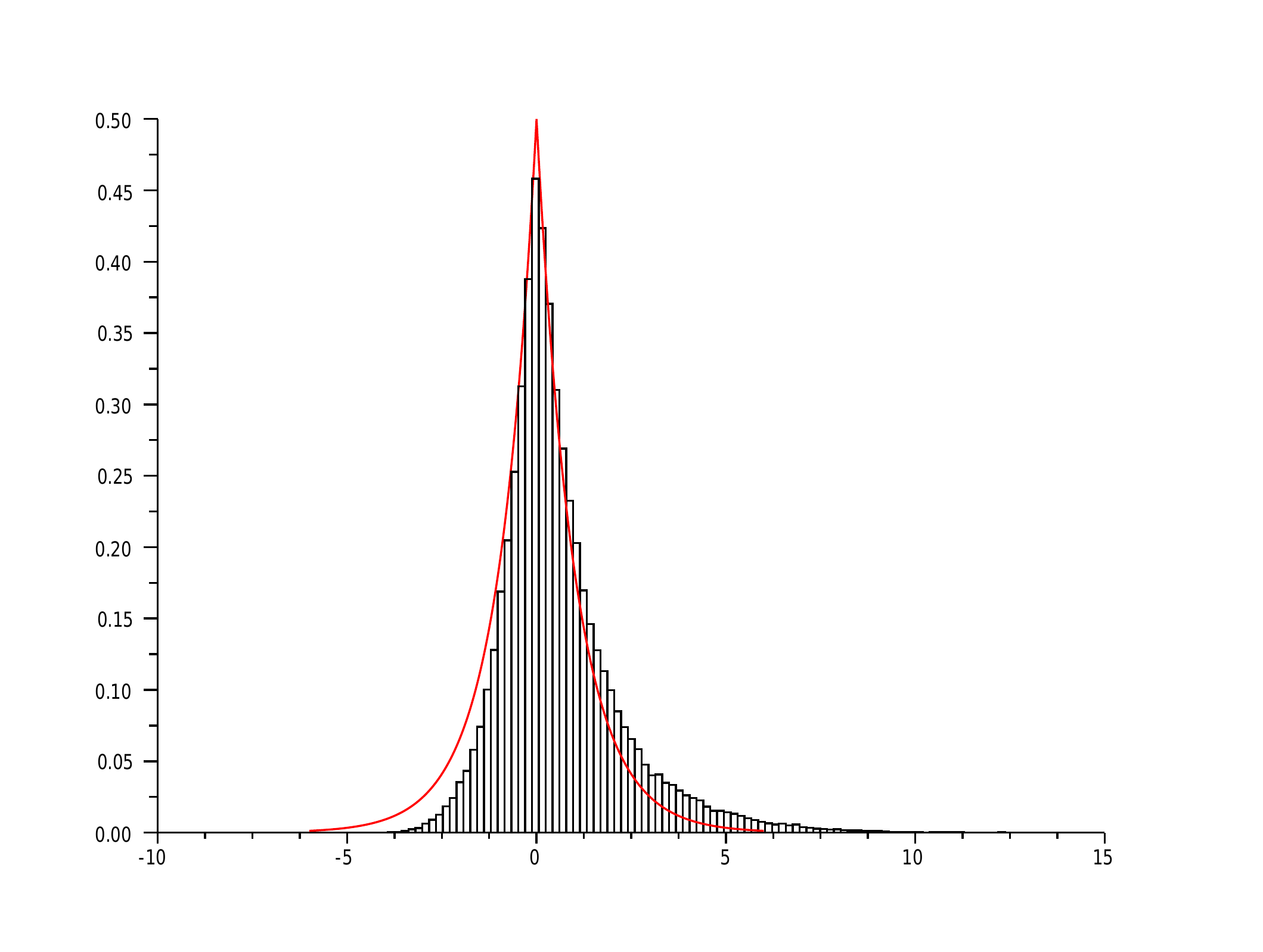}
  \caption{Empirical law of $Y_t$ starting at $(5,-1)$ for
  $t\in\BRA{2,6,10,14,18,22}$ with $a=1$ and $b=2$.}
   \label{fi:loivrai}
\end{center}
\end{figure}

In spite of the simple form of the process \eqref{eq:geninfsimple}, 
fully explicit computations on this model are not easy  to carry out.  When $Y_0=y>0$,  as long  as $t<y$  the law of $(Z_s)_{0\leq s\leq t}$  is  equal to  the 
law of the process with generator
 \begin{equation}\label{eq:tele-herr} 
Hf(y,w)=w\partial_yf(y,w)+\PAR{a+(b-a)\ind_\BRA{w>0}}(f(y,-w)-f(y,w),
\end{equation} 
which was computed in \cite{herrmann} in terms of modified Bessel 
functions. We have been unable to compute the transition laws for
\eqref{eq:geninfsimple} for general time intervals. (Notice when $a<b$  
that long-time behavior of the process driven by \eqref{eq:geninfsimple} 
is completely different  from that  of the process driven by \eqref{eq:tele-herr},
which  drifts to $-\infty$.) 

The proof of  the bound in Theorem \ref{th:no-reflection} will rely on 
the construction of a  coupling (see   \cite{lindvall} and \cite{Thori} for 
background),  which classically provides a convergence rate to 
equilibrium depending on tail estimates of the coupling time. A  related 
and popular approach to  the long-time behavior of Markov processes is 
the Foster-Lyapounov-Meyn-Tweedie theory (see 
\cite{MT,MT2,RR96,BCG}), which allows one to prove exponential ergodicity 
under conditions that are relatively easy to check. Specific applications 
 to  PDMP have been  developed in \cite{CD1,CD2}. Such 
 general results however provide  convergence estimates which  
are not fully explicit,  and sharpness of the bound and rates that can 
be deduced is hardly assessable.  

A fundamental step  in proving Theorem \ref{th:no-reflection} will be 
to first establish  an analogous result for the reflected (at the origin) 
version of the process. A fully explicit coupling will be first constructed  for 
the latter, inspired in the coalescent couplings for 
classic non-negative  continuous-time Markov process,  namely the M/M/1  queue and reflected Brownian motion with 
negative drift. The main difficulties in our case are that we have 
to couple both position and velocity and that,  contrary to those 
examples, we do not have a natural order structure. This 
 prevents us from  using  the   framework developed in  \cite{LMT}  to construct 
couplings  for processes  that are said to be stochastically ordered.

We will next recall basic ideas employed to study the long time behavior 
of Markov processes {\it via} couplings, following \cite{lindvall}. We 
also introduce the ``reflected version'' of the process driven 
by \eqref{eq:geninfsimple}  and state  an  analogue  of 
Theorem \ref{th:no-reflection} for the latter in Theorem \ref{th:conv-reflection}. The 
strategy of the proofs of Theorems \ref{th:no-reflection} and 
\ref{th:conv-reflection} together with the structure of the remainder 
of the paper are then explained. Let us anticipate  that the convergence rate  $\lambda_c$ in 
\eqref{eq:var-no-reflection} will arise as the supremum of the 
domain of the Laplace transform of the hitting times of origin for the 
process $(Y_t)$, suggesting that this rate is sharp. In this
direction, we also will see below  that in the suitable scaling limit where $(Y_t)$ 
converges to the Brownian motion drifted to the origin, the known  total variation convergence 
rate to equilibrium  of the latter is recovered as the rescaled limit of 
the $\lambda_c$'s.

\subsection{Preliminaries }\label{sec:prelim}
 In all the sequel we will use the notation 
$\overset{\cL}{=}$ meaning ``equal in law to''. By $\cE(\lambda)$  
and $\cP(\lambda)$ we will respectively denote the exponential law and 
the Poisson law of parameter $\lambda>0$,  whereas ${\cal B}(p)$ will stand 
for the Bernoulli distribution of parameter $p\in (0,1)$.

Recall that the total variation distance between two probability measures
$\eta$ and $\tilde \eta$ in a measurable space $\cX$ is given by
\begin{equation}\label{eq:def-var-tot}
\NRM{\eta-\tilde \eta}_{\mathrm{TV}}=
\inf\BRA{\dP(X\neq \tilde X)\, :\, X,\tilde X 
\mbox{ random variables  with }\cL(X)=\eta,\ \cL(\tilde X)=\tilde \eta},
\end{equation}
where each pair of random elements  $(X,\tilde X)$ of $\cX$ is 
simultaneously constructed in some probability space and is called 
a {\it coupling} (see \cite{lindvall} for alternative definitions of this
distance and its main properties). A coupling  $(U_t, \tilde{U}_t)_{t\geq 0}$ 
of two stochastic processes such that    $U_{t+T_*}=\tilde{U}_{t+T_*}$ for 
any $t\geq 0$ and an almost surely finite random time $T_*$  is called a 
{\it coalescent coupling} ($T_*$ is  then called a {\it coupling time}). It 
follows in this case that
$$
\NRM{\cL(U_t)-\cL(\tilde{U}_t)}_\mathrm{TV}\leq \dP (T_*>t).
$$ 
A helpful notion in obtaining an  effective control of the distance is 
\emph{stochastic domination}:

\begin{defi}[\cite{lindvall}]
 Let $S$ and $T$ be two non-negative random variables with
 respective cumulative distribution functions $F$ and $G$. We
 say that $S$ is stochastically smaller than $T$ and we write 
 $S\leq_\mathrm{sto.}T$, if   $F(t)\geq G(t)$ for any $t\in\dR$. 
\end{defi}


In particular, for a couple $(U_t, \tilde{U}_t)$ as above,  Chernoff's inequality  yields
\begin{equation}\label{coupleineg}
\NRM{\cL(U_t)-\cL(\tilde{U}_t)}_\mathrm{TV}\leq \dP (T>t) \leq  
\dE\PAR{e^{\lambda T}} e^{ -\lambda t}
\end{equation}
for any non-negative random variable $T$ such that 
$T_*\leq_\mathrm{sto.}T$, and any $\lambda\geq 0$ in the domain 
of the Laplace transform   $\lambda \mapsto \dE\PAR{e^{\lambda T}}$ of $T$.  

We will use these ideas to obtain  the exponential convergence estimates 
for $Z=(Y,W)$   in Theorem \ref{th:no-reflection},  and in Theorem 
\ref{th:conv-reflection} below  for its reflected  version  $(X,V)$ which 
we now introduce. The Markov process ${((X_t,V_t))}_{t\geq 0}$ 
is defined by its infinitesimal generator:
\begin{equation}\label{eq:giabs}
Af(x,v)=v\partial_xf(x,v)+\PAR{a+(b-a)\ind_\BRA{v>0}
+\frac{\ind_\BRA{x= 0} }{\ind_\BRA{x> 0}}}(f(x,-v)-f(x,v)),
\end{equation}
with $0<a<b$ (the term $\ind_\BRA{x= 0}( \ind_\BRA{x> 0} )^{-1}$ 
means that $X$ is reflected at zero). The dynamics of the process
 is simple: when $X$ is increasing (resp. decreasing), $V$
 flips to $-V$ with rate $b$ (resp. $a)$ and it is reflected in
 the origin (\emph{i.e.} as soon as $X=0$, $V$ flips to 1). 
 Given a path ${((Y_t,W_t))}_{t\geq 0}$ driven by
\eqref{eq:geninfsimple},  a path of
${((X_t,V_t))}_{t\geq 0}$  can be constructed taking
  $$X_t=\ABS{Y_t}, \quad 
V_0=\mathrm{sgn }(Y_0)W_0$$ and  defining the set of
jump times of $V$ to be 
$$
\BRA{t>0\,:\, \Delta V_t \neq 0}=%
\BRA{t>0\,:\, \Delta W_t \neq 0}\cup \BRA{t>0\,:\,Y_t=0}.
$$
Notice that since $W$ does not jump with positive probability when
$Y$ hits the origin,  one can also construct a path of 
${((Y_t,W_t))}_{t\geq 0}$  from an initial value
$y\in\dR$ and a path ${((X_t,V_t))}_{t\geq 0}$ driven by
\eqref{eq:giabs}: writing  $\sigma_0=0$ and   ${(\sigma_i)}_{i\geq 1}$ for the successive
hitting times   of the origin, we  define
$$
(Y_t,W_t)=(-1)^{i}\mathrm{sgn}(y) (X_t,V_t)   \quad\text{if }t\in [\sigma_{i},\sigma_{i+1}].
$$
Let us state our results about the long time behavior of $(X,V)$. 
\begin{thm}\label{th:conv-reflection}
 The invariant measure of $(X,V)$ is the product
 measure on $\dR_+\times\BRA{-1,+1}$ given by
$$
\nu(dx,dv)=(b-a)e^{-(b-a)x}\,dx\otimes  \frac{1}{2}(\delta_{-1}+\delta_{+1})(dv).
$$
If $\nu_t^{x,v}$ stands for the law of $(X_t,V_t)$ when $X_0=x$ and $V_0=v$,
we have, for any $x,\tilde x\geq 0$ and $v,\tilde v\in\BRA{-1,+1}$,
\begin{equation}\label{eq:up-reflection}
\NRM{\nu_t^{x,v}-\nu_t^{\tilde x,\tilde v}}_{\mathrm{TV}}\leq %
 \frac{(a+b)b}{2a^2}e^{r(a,b)(x\vee\tilde x)} e^{-\lambda_c t},
 \end{equation}
 where
$$
r(a,b)=\frac{3(b-a)}{4} \vee (b-\sqrt{ab})
\quad\text{and}\quad
\lambda_c=\frac{a+b}{2}-\sqrt{ab}=\frac{(\sqrt{b}-\sqrt{a})^2}{2}.
$$
\end{thm}

Notice that for small times $t\leq \ABS{x-\tilde x}/2$ the total variation 
distance does not decrease exponentially fast: the distance between  
$\nu_t^{x,v}$ and $\nu_t^{\tilde x,\tilde v}$ is equal to 1 since 
the supports of these two probability measures are disjoint.    

Theorem \ref{th:conv-reflection} should be compared to results on  two classic examples of ergodic 
non-negative continuous time Markov processes,  obtained  by coupling
arguments that are  briefly recalled next. Consider first  Brownian motion with negative drift 
$-c<0$ reflected at the origin and  which  has  the law $\cE(2c)$ as
invariant measure  (see \cite{lindvall} for this and the following 
facts). A coupling of two of its copies $(U^x_t)_{t\geq 0}$  and 
$(U_t^{\tilde{x}})_{t\geq 0}$  respectively starting from 
$x$ and $\tilde{x}$ consists in letting them evolve 
independently until they are equal for the first time and choosing  them  equal from that moment on. By non-negativity and continuity the coupling time $T_*$ 
for  $(U^x_t, U_t^{\tilde{x}})_{t\geq 0}$ is stochastically smaller that 
the hitting   time $T$  of the origin for   ${(U_t^{x\vee \tilde x})}_{t\geq 0}$.
Since for   $y>0$  the   hitting time $T$  of the origin by $(U^y_t)_{t\geq 0}$   satisfies
$\dE_y\PAR{e^{\lambda T}}=\exp\PAR{y(c-\sqrt{c^2-2\lambda})}$  
if $\lambda\in (-\infty, c^2/2]$ and $\dE_y\PAR{e^{\lambda T}}=+\infty$ otherwise 
(see e.g.  \cite[p. 70]{revuz-yor}), taking $\lambda=c^2/2$ in 
\eqref{coupleineg} one gets
\begin{equation}\label{rateRBM}
  \NRM{\cL(U^{x}_t)-\cL( U^{\tilde x}_t)}_{\mathrm{TV}}\leq 
  e^{c(x\vee\tilde x)}  e^{-c^2t/2}  \quad\mbox{ for all }\,  x,\tilde x\in\dR_+, \, t\geq 0 .
 \end{equation}
This estimate can also  be used to study the long time behavior of 
the solution of the  SDE
\begin{equation}\label{eq:bang}
d\xi_t=dB_t-c \ \mathrm{sgn}( \xi_t)\,dt
\end{equation}
which has the Laplace law $ c e^{-2c\vert x\vert}dx$ as  invariant measure. 
One  first has to couple the absolute values; the  first hitting time 
of the origin after their coupling time stochastically  dominates the  coupling
time for \eqref{eq:bang}.  A second example is  the M/M/1
queue, that is the continuous time Markov process ${(N_t)}_{t\geq 0}$ 
taking values in  $\dN$ with infinitesimal generator
 $$
 \tilde A f(n)=a(f(n+1)-f(n))+b\ind_\BRA{n>0}(f(n-1)-f(n))
 $$ 
where $b>a>0$ (to ensure ergodicity). Since two independent copies 
of the process starting from $n$ and $\tilde n$ do not jump simultaneously
and they have  one unit long jumps, their coupling time is smaller 
than the hitting time of the origin $T$ for the process starting at 
$n\vee \tilde n$.  For each initial state $n\in \dN$ the Laplace transform of $T$  has domain 
$(-\infty, (\sqrt{b}-\sqrt{a})^2]$  and  we have
$\dE_n\PAR{e^{(\sqrt{b}-\sqrt{a})^2 T}}=
\PAR{\frac{b}{a}}^{n/2}$ (see \cite{robert} for these facts)
which as before yields 
  $$
  \NRM{\cL(N^{n}_t)-\cL(\tilde N^{\tilde n}_t)}_{\mathrm{TV}}\leq 
  \PAR{\frac{b}{a}}^{(n\vee \tilde n)/2}  e^{-(\sqrt{b}-\sqrt{a})^2 t}
  \quad \mbox{ for any }n,\tilde n\in\dN,  \, t\geq0. 
  $$
We notice that in the appropriate scaling limit,  the M/M/1 queue  is  furthermore known to  converge to
 the reflected Brownian motion with negative drift (see \cite{robert}).

The construction of  a coalescent coupling for the process $(X,V)$ 
driven by \eqref{eq:giabs} is  harder than the previous examples 
since both positions and velocities must be coupled at some time. 
This will be  done in two steps. In Section \ref{subsec:crossing} 
we will obtain an estimate (in the sense of stochastic domination) 
for the first crossing time and position of $X$ and $\tilde X$ for a suitable coupling 
of the pair. At that time the velocities will be different. We will then construct 
in Section~\ref{subsec:stick} the coalescent coupling when starting 
from that special configuration. In Section~\ref{subsec:coupl} we will 
obtain an explicit upper bound for the Laplace transform of the 
coalescent time,  and thus the  quantitative convergence bound 
\eqref{eq:var-no-reflection}. The required stochastic dominations 
will be established  in terms of hitting times and lengths of  excursions  of $(X,V)$ 
away from the origin. These hitting times   will be  previously studied in 
Section~\ref{sec:absolute}. We will also give therein a complete 
description of the  excursions and compute  thereby the  invariant measure  of $(X,V)$ using a standard 
regeneration argument. Finally, Theorem~\ref{th:no-reflection} will be 
proved in Section~\ref{sec:no-reflection}  by  transferring these results 
to the unreflected process.

 We end this section noting that $\lambda_c$  in \eqref{eq:var-no-reflection} 
is the right convergence rate for the process \eqref{eq:geninfsimple}, at least  in the natural  diffusive asymptotics of the process. Let  $c>0$ and 
$0<a_\N<b_\N, N\in \dN$ be real numbers. We have

\begin{prop}\label{difulimit}
Assume that  $a_\N+b_\N\to \infty$ and $b_\N-a_\N\to 2c\in (0,\infty)$ 
as $N \to \infty$.  Let  $(Y^{(\N)}_t, W^{(\N)}_t)_{t\geq 0}$ denote the 
process driven by \eqref{eq:geninfsimple} with coefficients $a=a_\N,b=b_\N$ 
and starting from a random variable $Y_0^{(\N)}=\xi_0\in \dR$. Then, as 
$N\to \infty$,  the process
$$
\left(\xi^{(\N)}_t\right)_{t\geq 0}:=\left(Y^{(\N)}_{t(a_\N+b_\N)/2}\right)_{t\geq 0}
$$   
converges in  law in $C([0,\infty), \dR)$  to the solution $\xi_t$ of the 
SDE \eqref{eq:bang} with initial condition~$\xi_0$.
\end{prop}

\begin{rem}
By Theorem \ref{th:no-reflection} and the dual representation of the 
total variation distance, 
$$
| \dE(f(\xi^{(\N),y}_t)- \dE(f(\xi^{(\N),\tilde y})) | \leq 
C(a_\N,b_\N) 
e^{r(a_\N,b_\N) (\ABS{y }\vee \ABS{ \tilde y}) } 
\exp \left\{- (a_\N+b_\N) \frac{(\sqrt{b_\N}-\sqrt{a_\N})^2}{4}  t\right\}
$$
holds for every $t>0$ and each continuous function $f:\dR \to [-1,+1]$. 
Letting $N\to \infty$ and  taking then supremum over even functions $f$ 
in the previous class,  we then get that 
\begin{equation*}  \NRM{\cL(U^{x}_t)-\cL( U^{\tilde x}_t)}_{\mathrm{TV}}
\leq e^{3c(x\vee\tilde x)/2}  e^{-c^2t/2}  
\quad\mbox{ for all }\,  x,\tilde x\in\dR_+, \, t\geq 0,
 \end{equation*}
where $(U^x_t)_{t\geq 0}$  and 
$(U_t^{\tilde{x}})_{t\geq 0}$    are Brownian motions with drift 
$-c<0$ reflected at the origin,  respectively starting from 
$x$ and $\tilde{x}$.  
Comparison with  \eqref{rateRBM} suggests that the convergence rate 
$\lambda_c$ of Theorem \ref{th:no-reflection} cannot be substantially 
improved on,  and that one could in principle  improve upon   the exponent  $r(a,b)$ therein (more precisely,  upon  the term $3(b-a)/4$). \end{rem}

\begin{proof} We will use a standard diffusion approximation argument. 
Omitting  for a moment the sub and superscripts for notational 
simplicity, and writing 
$j_t:=  W_t -2\kappa W_t (a+(b-a)\ind_\BRA{Y_t W_t>0}) $,   
$J_t: =\int_0^t j_s ds$ and $\hat{Y}_t:=Y_t+ \kappa W_t$ for a given 
constant $\kappa >0$, we see by Dynkin's theorem that the processes 
$M_t:= \hat{Y}_t-J_t $ and 
$N_t:= \hat{Y}_t^2-t 2\kappa  -  \int_0 ^t 2Y_s j_s ds$ are local 
martingales with respect to the  filtration generated by $(Y_t,W_t)$. 
Using integration by parts we then get that 
$M_t^2= N_t - 2\int_0^t J_{s-} dM_s +2\kappa t -2\kappa\int_0^t W_s j_s ds$.  
Thus, noting that  
$j_s=\mathrm{sgn}(Y_s) \left( (2a\kappa -1)  
+2\times \ind_\BRA{Y_sW_s>0}(1-\kappa (b+a))\right)$,  
we see for $\kappa=(a+b)^{-1}$ that 
$$
M_t=Y_t- \left[ \int_0 ^t \mathrm{sgn}(Y_s) 
\left(\frac{a-b}{a+b}\right)ds  -\frac{W_t}{a+b}\right] \, , 
\, M_t^2- \left[\frac{2t}{a+b}  -2\int_0^t W_s\, 
\mathrm{sgn}(Y_s)  \frac{a-b}{(a+b)^2}  ds  \right]
$$
are local martingales. Therefore,  defining for each $N\in \dN$ 
$$
\beta^{(\N)}_t:=  \frac{(a_\N- b_\N)}{2} \int_0^t 
\mathrm{sgn}(\xi_s^{(\N)})ds - \frac{W^{(\N)}_{t(a_\N+b_\N)/2}}{a_\N+b_\N}
$$
and 
$$
\alpha^{(\N)}_t:= t- \frac{(a_\N-b_\N)}{a_\N+ b_\N} 
\int_0^t W^{(\N)}_{s(a_\N+b_\N)/2}\,  \mathrm{sgn}(\xi_s^{(\N)})ds \, , 
$$
we readily check that the processes $\xi^{(\N)}_t, \alpha^{(\N)}_t$ and 
$\beta^{(\N)}_t$ satisfy assumptions (4.1) to (4.7) of Theorem 4.1 in 
\cite[p. 354]{ethier} (in the respective roles of the processes $X_n(t),A_n(t)$ 
and $B_n(t)$ therein). That result ensures  that  $\cL(\xi^{(\N)})$ 
converges weakly to the unique solution of the martingale problem with 
generator $Gf(x):= \frac{1}{2}f''(x) -c\, \mathrm{sgn}(x)f'(x), 
f\in C_c^{\infty}(\dR) $ and initial law $\cL(\xi_0)$. 
\end{proof}

\section{The invariant measure of the reflected process}\label{sec:absolute}

In this section  we will determine the invariant measure of $(X,V)$.
This process is clearly positive recurrent since, as will be shown in
the sequel, the Laplace transform of the hitting time of $(0,+1)$ is
finite on a neighborhood of the origin, whatever the initial data are. 
We will need the following well-known results for Poisson processes.

\begin{prop}[\cite{norris}]\label{prop:poisson}
Let ${(N_t)}_{t\geq 0}$ be a Poisson process with intensity
 $\lambda>0$. Denote by ${(T_n)}_{n\geq 1}$ its jump times. 
Then,  $N_t\sim\cP(\lambda t)$ for any $t\geq 0$. Moreover,
conditionally on  $\BRA{N_t=k}$, the jump times
$T_1,T_2,\ldots, T_k$ have the same distribution than an
ordered sample of size $k$ from the uniform distribution
on $[0,t]$.
\end{prop}

\subsection{Excursion and hitting times}

We start by computing  the Laplace transforms of the length 
of an excursion (to be defined next) and of the hitting times 
of the origin when starting from $(x,v)\in\dR_+\times \BRA{-1,+1}$.
\begin{defi}
 An excursion of $(X,V)$ driven by \eqref{eq:giabs} is a
 path starting at $(0,+1)$ and stopped at
 $$
 S=\inf\BRA{t>0\, :\, X_t=0}.
 $$
We  denote by $\psi$ the Laplace transform of $S$:
 \begin{equation}\label{eq:def-psi}
\psi\,:\,\lambda\in\dR\mapsto \psi(\lambda)=\dE_{(0,+1)}\PAR{e^{\lambda S}}.
 \end{equation}
\end{defi}
Notice that $\lim_{t\rightarrow S^-}V_t=-1$ and $V_S=1$.

\begin{lem}[Length of an excursion.]\label{le:laplaceS}
 The domain of $\psi$ defined in \eqref{eq:def-psi} is equal
  to $(-\infty,\lambda_c]$ where
\begin{equation} \label{eq:lc}
\lambda_c= \frac{a+b}{2}-\sqrt{ab}=\frac{(\sqrt{b}-\sqrt{a})^2}{2}.
\end{equation}
Furthermore, if $\lambda\leq \lambda_c$,
\begin{equation}\label{eq:psi}
\psi(\lambda)=
\frac{a+b-2\lambda-\sqrt{(a+b-2\lambda)^2-4ab}}{2a}.
\end{equation}
In particular, $\psi(\lambda_c)=\sqrt{b/a}$ and $\dE_{(0,+1)}(S)=2/(b-a)$.
\end{lem}

\begin{proof}
 During a time length $E$ of
law $\cE(b)$, $V$ is equal to 1 and $X$ grows linearly. At time $t=E$,
 $V$ flips to $-1$ and $X$ starts going down. Denote
 by $T_2$ the second jump time of $V$. If $X_{T_2}=0$,
then $S=2E$. Otherwise, $X$ starts a new excursion
above $X_{T_2}$ which has the same law as $S$ and 
is independent of the past. After this
 excursion, $(X,V)$ is equal to $(X_{T_2},-1)$. Once again, it
 reaches 0 directly or $V$ flips to 1 before doing so, in which case
 a new independent excursion begins. Proposition \ref{prop:poisson} and the 
 strong Markov property ensure that, conditionally on $\BRA{E=x}$, 
 the number $N$ of embedded excursions has the law $\cP(ax)$. 
We thus can decompose $S$ as 
 $$
S\overset{\cL}{=}2E+\sum_{k=1}^NS_k,
$$
where  $E\sim \cE(b)$, $\cL(N\vert E=x)=\cP(ax)$ and
${(S_k)}_{k\geq 1}$ is an  i.i.d. sequence of  random
variables distributed as $S$ and independent of the couple
$(E,N)$. As a consequence,
\begin{align*}
\psi(\lambda)&=
\dE\PAR{\dE\PAR{e^{2\lambda E+\lambda\sum_{k=1}^NS_k}\vert E,N}}
=\dE\PAR{e^{2\lambda E} \dE\PAR{\psi(\lambda)^N\vert E}}\\
&=\dE\PAR{e^{2\lambda E}e^{aE(\psi(\lambda)-1)}}
=\frac{b}{b+a-2\lambda-a\psi(\lambda)}
\end{align*}
for each $\lambda$ in the domain of $\psi $ (which contains $(-\infty,0]$).
 This implies (since $\psi$ is a Laplace transform) that
$$
\psi(\lambda)=\frac{a+b-2\lambda-\sqrt{(a+b-2\lambda)^2-4ab}}{2a}.
$$
The relation is in fact valid as soon as the argument of the 
square root is non-negative
\emph{i.e.} as soon as $\lambda\leq \lambda_c$ with $\lambda_c$
defined in \eqref{eq:lc}. At last,
$\dE_{(0,+1)}(S)=\psi'(0)=2/(b-a)$.
\end{proof}

\begin{rem}[Number of jumps in an excursion]

Since each 
 excursion is preceded by a jump, the number $M$  
 of jumps of $V$ during an excursion  (omitting the 
 jump at time $S$)  satisfies 
 $$
M\overset{\cL}{=} 1+\sum_{i=1}^N(1+M_i)
$$
where 
 ${(M_i)}_{i\geq 0}$ is an i.i.d.  sequence with the same law as
 $M$ and  independent of the random variable $N$ such that $\cL(N\vert E)=\cP(aE)$ with 
 $E\sim\cE(b)$.  By conditioning first in $E,N$ as in the previous proof, one can easily  derive a second degree equation and then an explicit expression for the Laplace transform of the number of jumps. We omit the details since this result will not be needed.

\end{rem}

\begin{lem}
 For  $x> 0$,  let  $S_{(x,-1)}$ denote the hitting time of 0 starting 
 from $(x,-1)$. Then
 \begin{equation}\label{eq:cl}
 \dE\PAR{e^{\lambda S_{(x,-1)}}}=e^{xc(\lambda)}
\quad\text{with} \quad
c(\lambda)=\frac{b-a-\sqrt{(a+b-2\lambda)^2-4ab}}{2} 
\end{equation}
if $\lambda \in (-\infty,\lambda_c]$,  and $+\infty$ otherwise. 
\end{lem}

\begin{proof}
 As in the proof of Lemma \ref{le:laplaceS}, one can decompose $S_{(x,-1)}$ as
 $$
 S_{(x,-1)}\overset{\cL}{=}x+\sum_{k=1}^N S_k,
 $$
 where $N$ is a random variable with law $\cP(ax)$ independent of 
 the i.i.d. sequence of  random variables ${(S_k)}_{k\geq 1}$ with 
 Laplace transform $\psi$. Then,
 $$
 \dE\PAR{e^{\lambda S_{(x,-1)}}}
 =\sum_{k\geq 1}\dE\PAR{e^{\lambda (x+S_1+S_2+\cdots+S_k)}
 \ind_\BRA{N=k}}
 =e^{ax(\psi(\lambda)-1)+\lambda x}.
$$
At last, $a(\psi(\lambda)-1)+ \lambda$ is equal to $c(\lambda)$.
\end{proof}

\begin{cor}
 For any $x\geq 0$, let us denote by $S_{(x,+1)}$ the hitting 
 time of 0 starting from $(x,+1)$. Then
$$
 \dE\PAR{e^{\lambda S_{(x,+1)}}}= \psi(\lambda)e^{xc(\lambda)},
 $$
 where $\psi$ is given by \eqref{eq:psi}  and $c(\lambda)$ by \eqref{eq:cl}.
\end{cor}
\begin{proof}
The strong Markov property implies that $S_{(x,+1)}\overset{\cL}{=}S+S_{(x,-1)}$ 
where $S$ is the length of an excursion independent from $S_{(x,-1)}$.
\end{proof}

\begin{lem}\label{lem:somme}
 For any $x,\tilde x\geq 0$,
 $$
 S_{(x+\tilde x,-1)}\overset{\cL}{=} 
 S_{(x,-1)}+ S_{(\tilde x,-1)}\geq_{\mathrm{sto.}} S_{(x,-1)},
 $$
 where $S_{(x,-1)}$ and $S_{(\tilde x,-1)}$ are independent.
\end{lem}
\begin{proof}
 This is a straightforward consequence of the strong  Markov property.
\end{proof}

\subsection{The invariant measure}

Recall that  the invariant law of $(X,V)$ is denoted by $\nu$.  

\begin{lem}\label{le:invmeas}
 For any bounded function $f:
\dR\times\BRA{-1,+1}\to \dR$, we have
$$
\int\! f\,d\nu=\frac{1}{\dE_{(0,+1)}(S)}\dE_{(0,+1)}\PAR{\int_0^S\!f(X_s,V_s)\,ds},
$$
where $S$ is the first hitting time of 0. 

\end{lem}

\begin{proof}  We will use a standard result on  regenerative processes  (see 
 Asmussen \cite[Chapter  VI]{Asmussen} for 
 background).
Let $(S_n)_{n\geq 1}$ denote the lengths of  the  consecutive  
excursions away from $0$,    $S_0:=S$ 
and $\Theta_n:=S_0 +\dots +S_n$. By the strong Markov property, 
$(\Theta_n)_{n\in \dN}$ is a renewal process,  for each $n\in \dN$ 
the post $\Theta_n$-process  
$(\Theta_{n+1},\Theta_{n+2},\dots, (X_{\Theta_n+t},V_{\Theta_n+t})_{t\geq 0})$ 
is independent of $(\Theta_0,\dots, \Theta_n)$, and  is equally  distributed for all $n\geq 1$. This means that 
$(X_t,V_t)_{t\geq 0}$ is a regenerative process with regeneration points
 $(\Theta_n)_{n \in \dN}$ and  cycle 
 length  corresponding to the length of an excursion.   The result is  immediate from 
 \cite[Theorem 1.2, Chapter VI]{Asmussen} and Lemma \ref{le:laplaceS}. 
\end{proof}

\begin{lem}\label{F}  Define, for a
 non negative function $g:\BRA{-1,+1}\to \dR$ and $\lambda\in\dR$,
 $$
 F:=\int_0^Se^{\lambda X_s}g(V_s)\,ds.
 $$
 Then, conditionally on $(X_0,V_0)=(0,+1)$, we have
 $$
F\overset{\cL}{=}
(g(1)+g(-1))\int_0^E \! e^{\lambda y}\,dy+
\sum_{i=1}^N
\int_0^{S^{(i)}}\!e^{\lambda(EU_{(i,N)}+ X^{(i)}_s)}g(V^{(i)}_s)\,ds, 
\mbox{ where }
$$ 
\begin{itemize}
\item  ${(U_i)}_{i\geq 0}$ is an sequence of independent uniformly distributed  
random variables on $[0,1]$, and for each   $n\geq 1$ 
$(U_{(1,n)},U_{(2,n)},\ldots,U_{(n,n)})$ is the re-ordered sampling of
$(U_1,U_2,\ldots,U_n)$;
\item   ${(X^{(i)}_t,V^{(i)}_t)}_{0\leq t\leq S^{(i)}}$ is a sequence of 
independent excursions; 
\item  $E\sim \cE(b)$,   $\cL(N\vert E=x)=\cP(ax)$, and the pair 
$(E,N)$ is independent of the all the previous  random variables. 
 \end{itemize}
\end{lem}
 
  \begin{proof}  
The argument has been already been given in the first part of the proof 
of Lemma~\ref{le:laplaceS}. We just  notice that  the $N$  independent 
embedded excursions therein occur at the heights
$(EU_{(N,N)},EU_{(N-1,N)},\ldots,EU_{(1,N)})$ (see  
Proposition~\ref{prop:poisson}).
\end{proof}
  
  \smallskip
  
We now are ready to compute $\nu$   which is the first point in 
Theorem~\ref{th:conv-reflection}. Since
$$
\dE\PAR{\int_0^{S^{(i)}}\!e^{\lambda(EU_{(i,N)}+X^{(i)}_s)}g(V^{(i)}_s)\,ds %
\Big\vert E,N,{(U_i)}_{i\geq 1}}
=e^{\lambda EU_{(i,N)}} \dE_{(0,+1)}(F),
$$
we get that
$$
 \dE_{(0,+1)}(F)=%
 \dE\PAR{\frac{(g(1)+g(-1))}{\lambda}(e^{\lambda E}-1)+
 \dE_{(0,+1)}(F)\sum_{i=1}^Ne^{\lambda EU_{(i,N)}}}.
$$
We have, for any $x>0$ and $k\in\dN$,
$$
 \sum_{i=1}^k\dE\PAR{e^{\lambda xU_{(i,k)}}}
 =\dE\PAR{\sum_{i=1}^k e^{\lambda x U_i}}
 = \frac{k}{\lambda x}(e^{\lambda x}-1),
$$
and then
$$
\dE\PAR{ \sum_{i=1}^Ne^{\lambda EU_{(i,N)}}\Big\vert E}=
\dE\PAR{ \frac{N}{\lambda E}(e^{\lambda E}-1)\Big\vert E}=
 \frac{a}{\lambda }(e^{\lambda E}-1).
$$
As a conclusion, we have
\begin{align*}
\dE_{(0,+1)}(F)&=\frac{g(1)+g(-1)+a\dE_{(0,+1)}(F)}{\lambda} 
\PAR{\frac{b}{b-\lambda}-1}\\
&=\frac{g(1)+g(-1)}{b-\lambda}+\frac{a}{b-\lambda}\dE_{(0,+1)}(F),
\end{align*}
which provides the expression of $\dE_{(0,+1)}(F)$ for any $\lambda<b-a$:
$$
\dE_{(0,+1)}(F)=\frac{g(1)+g(-1)}{b-a-\lambda }.
$$
On the other hand, Lemma~\ref{le:laplaceS} ensures that 
$\dE_{(0,+1)}(S)=2/(b-a)$ from where we get, for any $\lambda<b-a$, that 
$$
\int_{\dR\times\BRA{-1,+1}}\! e^{\lambda x}g(v)\nu(dx,dv)=
\frac{b-a}{b-a-\lambda } \frac{g(1)+g(-1)}{2}.
$$
In other words, the invariant measure of $(X,V)$ is
$\nu=\cE(b-a)\otimes (1/2)(\delta_{-1}+\delta_{+1})$.

\section{The coalescent time for the reflected process}\label{sec:coalescent-refl}

\subsection{The crossing time}\label{subsec:crossing}

We will  first construct a coupling $(X,V,\tilde X,\tilde V)$ starting 
at $(x,v,\tilde x,\tilde v)$ until a time $T_c=T_c(x,v,\tilde x,\tilde v)$ 
called \emph{crossing time},  at which $X_{T_c}=\tilde X_{T_c}$. In 
doing so, we will also stochastically  control $T_c$ and $X_{T_c}$. 
The coupling will consist in making the two velocities equal as longer 
as possible.  Assume without loss of generality  that $\tilde x< x$. 
Plainly, if $V$ and $\tilde V$ are different, we let the two processes 
evolve independently until one of them performs a jump or until $X-\tilde X$ 
hits $0$. At that time, if $X\neq \tilde X$, the two velocities are equal and 
we set them equal until $\tilde X$ hits the origin. During this period the 
paths of $X$ and $\tilde X$ are parallel and, at the hitting time of the origin, 
$V$ and $\tilde V$ are once again different. We then iterate this procedure 
until $T_c$.  Notice that $\tilde X$ is smaller than $X$ on $[0,T_c)$. 
We now make the construction with full  details.

\subsubsection{The main initial configuration}

Assume first that $(X_0,V_0,\tilde X_0,\tilde V_0)=(x,-1,0,+1)$. The 
coupling works as follows: with rate $a$ (resp. $b$), $V$ (resp. $\tilde V$) 
flips to $+1$ (resp. $-1$) and if none of these two events occurs before 
time $x/2$, then
$$
X_{x/2}=\tilde X_{x/2}
\quad\text{and}\quad
V_{x/2}=-1=-\tilde V_{x/2}.
$$
If a jump occurs at time $\tau_1<x/2$, then
$(X_{\tau_1},V_{\tau_1},\tilde X_{\tau_1},\tilde V_{\tau_1})=(x-\tau_1,U,\tau_1,U)$,
where $U=-1$ with probability $b/(a+b)$ ($\tilde V$ jumps before $V$) 
and $U=1$ with probability $a/(a+b)$. Then, $V$ and $\tilde V$ are 
chosen equal until $\tilde X$ hits 0 \emph{i.e.} during a time 
$S_{(\tau_1,U)}$ and
$$
(X_{\tau_1+S_{(\tau_1,U)}},V_{\tau_1+S_{(\tau_1,U)}},
\tilde X_{\tau_1+S_{(\tau_1,U)}},\tilde V_{\tau_1+S_{(\tau_1,U)}})
=(x-2\tau_1,-1,0,+1).
$$
Notice that $S_{(\tau_1,+1)}\overset{\cL}{=}S+S_{(\tau_1,-1)}$ where
$S$ is the length of an excursion independent of $S_{(\tau_1,-1)}$.
As a conclusion, if a jump occurs at time $\tau_1<x/2$, then the full process
$(X,V,\tilde X,\tilde V)$ is equal to $(x-2\tau_1,-1,0,+1)$ at time
$\tau_1+S_{(\tau_1,-1)}+BS$ where $B\sim\cB(a/(a+b))$.  
One has to iterate this procedure until $X-\tilde X$ hits 0.

Consider now a Poisson process ${(N(t))}_{t\geq 0}$ with intensity
$a+b$. We denote by ${(T_n)}_{n\geq 0}$ its jump times (with
$T_0=0$) and define ${(\tau_i)}_{i\geq 1}$ by $\tau_i=T_i-T_{i-1}$
for $i\geq 1$. The number of return times at 0 for $\tilde X$
before $T_c$ is distributed as $N(x/2)$ and the length of the
periods when $(X,V)$ and $(\tilde X,\tilde V)$ are independent are
given by $\tau_1$, $\tau_2$,\ldots,$\tau_{N(x/2)}$ and
$x/2-T_{N(x/2)}$. Then,
$$
T_c(x,-1,0,+1)\overset{\cL}{=}\sum_{i=1}^{N(x/2)}
\PAR{\tau_i+S_{(\tau_i,-1)}+B_iS^{(i)}}+x/2-T_{N(x/2)},
$$
where the law of ${(S^{(i)})}_{i\geq 1}$ is the one of the length
of an excursion, ${(S_{(\tau_i,-1)})}_{i\geq 1}$ are the hitting
times of 0 starting from ${(\tau_i)}_{i\geq 1}$, ${(B_i)}_{i\geq
1}$ have the law $\cB(a/(a+b))$ and all these random variables are
independent. Since $T_{N(x/2)}=\tau_1+\tau_2+\cdots+\tau_{N(x/2)}$,
Lemma~\ref{lem:somme} ensures that
\begin{align}\label{eq:tc0}
T_c(x,-1,0,+1)&\overset{\cL}{=}x/2+S_{(T_{N(x/2)},-1)}
+\sum_{i=1}^{N(x/2)}\PAR{B_iS^{(i)}}\nonumber \\
&\leq_\mathrm{sto.}x/2+S_{(x/2,-1)}+\Sigma(x)
\end{align}
where
$$
\Sigma(x):\overset{\cL}{=}\sum_{i=1}^{N(x/2)}\PAR{B_iS^{(i)}}.
$$

Notice that  $\Sigma(u+v)\overset{\cL}{=}\Sigma(u)+\Sigma(v)$,
where $\Sigma(u)$ and $\Sigma(v)$ are independent, and that $\Sigma(u)$ distributes as the sum of the lengths of $N$ independent excursions, with $N\sim \cP(au/2)$.




\subsubsection{Other configurations}

We  next  construct the paths until $T_c(x,v,\tilde x,\tilde v)$ and 
control this time irrespective of the initial velocities. Without loss of 
generality, we can assume that $x\geq \tilde x$. We  just have to 
construct the paths until $(X,V,\tilde X,\tilde V)$ reaches a state 
$(u,-1,0,+1)$,  and then  make use of the previous section.

Assume firstly that $v=\tilde v=U\in\BRA{-1,+1}$. We have  to construct a
trajectory of $(\tilde X,\tilde V)$ until $S_{(\tilde x,U)}$ the hitting time of $0$.
Define for any $t\in [0,S_{(\tilde x,U)})$, $V_t=\tilde V_t$,
$X_t=\tilde X_t-\tilde x+x$, $V_{S_{(\tilde x,U)}}=-1$ and
$X_{S_{(\tilde x,U)}}=x-\tilde x$. Using Lemma~\ref{lem:somme} and
\eqref{eq:tc0}, one has
\begin{align*}
T_c(x,U,\tilde x,U)&\overset{\cL}{=}S_{(\tilde x,U)}+T_c(x-\tilde x,-1,0,+1)\\
&\overset{\cL}{=}S\ind_\BRA{U=+1}+S_{(\tilde x,-1)}+T_c(x-\tilde x,-1,0,+1)\\
&\leq_\mathrm{sto.}S\ind_\BRA{U=+1}
+\frac{x-\tilde x}{2}+S_{((x+\tilde x)/2,-1)}+\Sigma(x-\tilde x).
\end{align*}



Assume now that $v=1=-\tilde v$. The processes $(X,V)$ 
and $(\tilde X,\tilde V)$ are chosen independent until the first 
jump time.  This  is equal to $E=(E_1\wedge \tilde x)\wedge E_2$, 
where $E_1\sim \cE(a)$, $E_2\sim \cE(b)$ and 
$(X_E,V_E,\tilde X_E,\tilde V)=(x+E,U,\tilde x-E,U)$ with 
$U\in\BRA{-1,+1}$. In particular, one has
$$
\frac{X_E+\tilde X_E}{2}=\frac{x+\tilde x}{2}%
\quad\text{and}\quad
\frac{X_E-\tilde X_E}{2}=\frac{x-\tilde x}{2}+E.%
$$
 Since for any $y,\tilde y\geq 0$, 
 $T_c(y,-1,\tilde y,-1)\leq_\mathrm{sto.} T_c(y,1,\tilde y,1)$ 
 this ensures that
\begin{align*}
T_c(x,+1,\tilde x,-1)&\leq_\mathrm{sto.} E+T_c(x+E,+1,\tilde x-E,+1)\\
&\leq_\mathrm{sto.} E+S+\frac{x-\tilde x}{2}+E+S_{((x+\tilde x)/2,-1)}
+\Sigma(x-\tilde x+2E)\\
&\leq_\mathrm{sto.} 2E+\Sigma(2E)+S+\frac{x-\tilde x}{2}
+S_{((x+\tilde x)/2,-1)}+\Sigma(x-\tilde x).
\end{align*}

If $v=-1=-\tilde v$, we proceed as in the previous case. With the same notations,
$$
\frac{X_E+\tilde X_E}{2}=\frac{x+\tilde x}{2}%
\quad\text{and}\quad
\frac{\vert X_E-\tilde X_E\vert}{2}\leq \frac{x-\tilde x}{2}+E.
$$
We then get the same upper bound as before. 
As a conclusion we have established the following upper bound
for $T_c$:
\begin{lem}\label{lem:tc}
For any $x\geq \tilde x$ and $v,\tilde v\in\BRA{-1,+1}$,
$$
T_c(x,v,\tilde x,\tilde v)\leq_{\mathrm{sto.}} 2E+\Sigma(2 E)+
S+S_{((x+\tilde x)/2,-1)}+\frac{x-\tilde x}{2}+\Sigma(x-\tilde x),
$$
where $E=F\wedge \tilde x$ with $F$ is an exponential variable
with parameter $a+b$ and $\Sigma(u)$ is the sum of the  lengths of $N$ 
independent excursions where $N\sim\cP(au/2)$. Moreover,
$$
X_{T_c}=\tilde X_{T_c}\leq \frac{x-\tilde x}{2}+E
\quad\text{and}\quad
V_{T_c}=-\tilde V_{T_c}.
$$
\end{lem}

\subsection{A simple way to stick the paths}\label{subsec:stick}

We now assume that $(X_0,V_0)=(x,1)$ and
$(\tilde X_0,\tilde V_0)=(x,-1)$ and construct  two paths which
are equal after a coalescent time $T_{cc}(x)$.  The idea is to use   
the same exponential clocks for both paths but in a different order. 
We explain the generic step of this construction considering $R$ 
and $Q$ two given independent random variables with respective 
laws $\cE(a)$ and $\cE(b)$. There are two possible situations:

\begin{figure}
\begin{center}
 \includegraphics[scale=0.6]{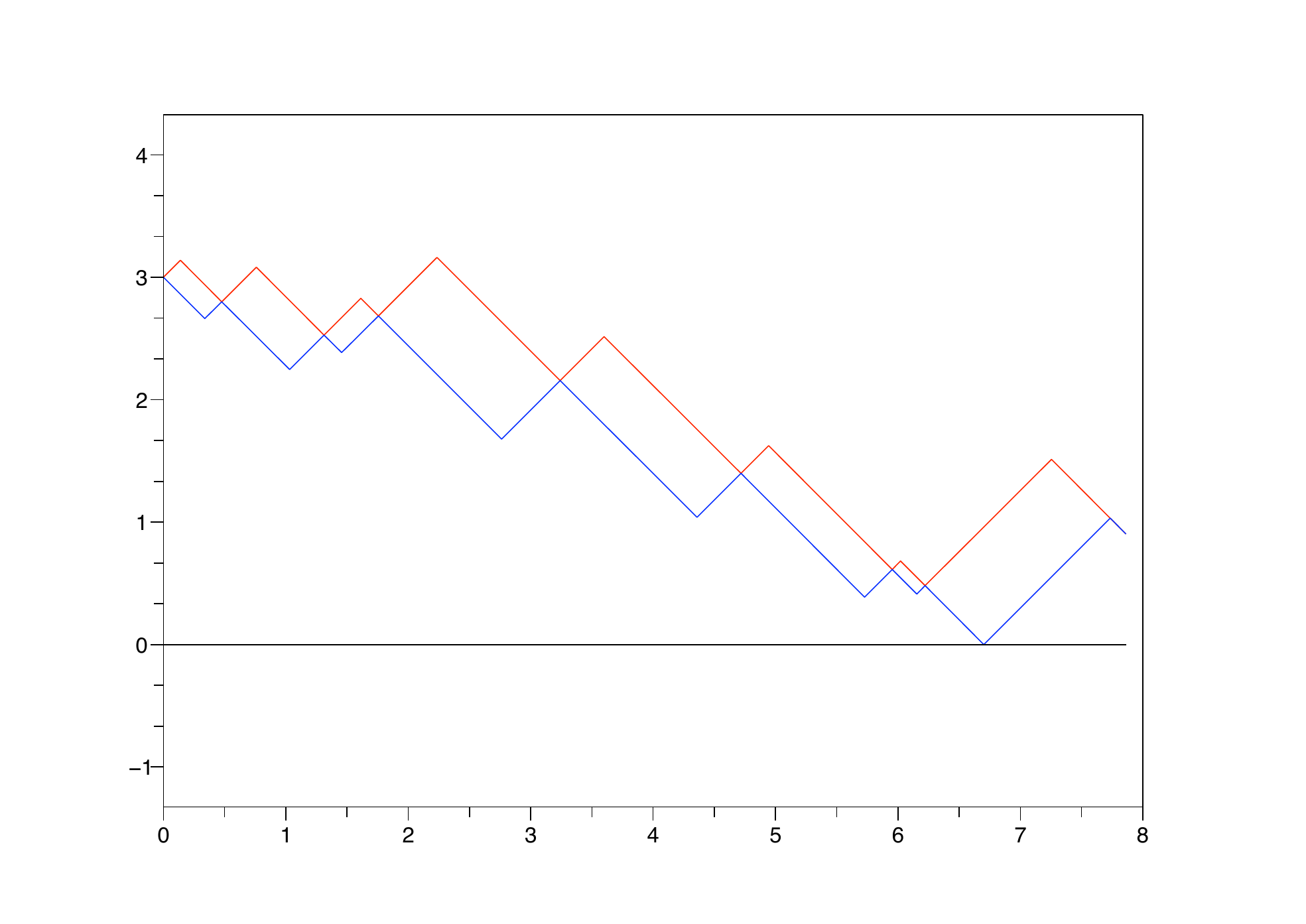}
 \caption{Two paths starting at $x=3$ with different velocities until
 they stick together.}
 \label{fi:colle}
\end{center}
\end{figure}

\begin{itemize}
 \item Case 1:  $R<x$.  In this case, defining $T=R+Q$,
 $$
 V_t=
 \begin{cases}
 +1 & \text{if }t\in[0,Q),\\
 -1 & \text{if }t\in[Q,T), \\
 +1& \text{if }t=T
 \end{cases}
 \quad\text{and}\quad
 \tilde V_t=
 \begin{cases}
 -1 & \text{if }t\in[0,R),\\
 +1 & \text{if }t\in[R,T),\\
 -1& \text{if }t=T,
  \end{cases}
$$
one has  $X_T=x+Q-R=\tilde X_T$ and
$V_T=1=-\tilde V_T$.  
 \item Case 2: $R\geq x$. In this case,  defining $T=x+Q$,
 $$
 V_t=
 \begin{cases}
 +1 & \text{if }t\in[0,Q),\\
 -1 & \text{if }t\in[Q,T),
 \end{cases}
  \quad\text{and}\quad
 \tilde V_t=
 \begin{cases}
 -1 & \text{if }t\in[0,x),\\
 +1 & \text{if }t\in[x,T),\\
 -1& \text{if }t=T,
  \end{cases}
$$
one has  $X_T=Q=\tilde X_T$ and $V_T=\tilde V_T=-1$.
In this case  $(X,V)$ and $(\tilde X,\tilde V)$ are coupled at time $T$.
\end{itemize}
We now  construct the paths. We take  an i.i.d. sequence of 
independent pairs of exponential variables $(R_n,Q_n)$ with 
$R_n \sim \cE(a)$ and $Q_n\sim \cE(b)$, and  inductively define 
$\tau_0=0$ and $\tau_{n+1}=\tau_n+T_n$,  with $T_n$  defined 
from $(R_n,Q_n)$ as above until Case 2 occurs. At each iteration, 
 \begin{itemize}
\item if $\tilde X$ does not hit the origin in the interval $[\tau_n,\tau_{n+1}]$
(Case 1) we set
$$
X_{\tau_{n+1}}=\tilde X_{\tau_{n+1}}
\quad\text{and}\quad
V_{\tau_{n+1}}=1=-\tilde V_{\tau_{n+1}} 	\, ; 
$$
\item if $\tilde X$ hits the origin in the interval $[\tau_n,\tau_{n+1}]$ (Case 2)
we set
$$
X_{\tau_{n+1}}=\tilde X_{\tau_{n+1}}
\quad, \quad
V_{\tau_{n+1}}=\tilde V_{\tau_{n+1}}=-1
\quad\text{and}\quad
T_{cc}(x):=\tau_{n+1}.
$$
\end{itemize}

By construction, $X_t\geq \tilde X_t$ for any $t\geq 0$ and the
coupling time $T_{cc}(x)$  is  smaller than the hitting time of the origin
time of $X$ (see Figure \ref{fi:colle}). As a conclusion we have shown 
the following result.
\begin{lem}\label{lem:tcc}
 There exists a coupling of $(X,V)$ and $(\tilde X,\tilde V)$ starting
 respectively from $(x,1)$ and $(x,-1)$ such that the coalescent
 time $T_{cc}(x)$ is (stochastically) smaller than $S_{(x,+1)}$ and
 $$
 X_{S_{(x,+1)}}=\tilde X_{S_{(x,+1)}}=0
 \quad\text{and}\quad
 V_{S_{(x,+1)}}=\tilde V_{S_{(x,+1)}}=1.
 $$
\end{lem}

\subsection{The Laplace transform of the coupling time}\label{subsec:coupl}

We  now gather the previous estimates to control the
Laplace transform of the coupling time of the two paths starting
respectively from $(x,v)$ and $(\tilde x,\tilde v)$:

\begin{prop}\label{prop:tb}
For any $x\geq \tilde x\geq 0$, any $v,\tilde v\in \BRA{-1,+1}$,
there exists a coalescent coupling such that the coupling time
$T(x,v,\tilde x,\tilde v)$ is stochastically smaller than a random
variable
\begin{equation}\label{eq:tb}
 \overline{T}(x,\tilde x)\overset{\cL}{=} F+\Sigma(2F)+S_{(F,-1)}
 +S+\tilde S+S_{(x,-1)}+\frac{x+\tilde x}{2}
 +\Sigma(x-\tilde x),
\end{equation}
where $F\sim\cE(a+b)$, $S$ and $\tilde S$ are excursion lengths
and all the random variables are independent. Furthermore, for
any $\lambda \in[0,\lambda_c]$,
$$
\dE\PAR{e^{\lambda T(x,v,\tilde x,\tilde v)}}\leq %
\frac{(a+b)\psi(\lambda)^2}{2a+b-\lambda-a\psi(\lambda)-c(\lambda)}
 \exp\PAR{x c(\lambda) +\frac{x+\tilde x}{2}\lambda
 +\frac{x-\tilde x}{2}a(\psi(\lambda)-1)}.
$$
At last, a realization of $\bar{T}(x,\tilde{x})$ is the first
hitting time at $0$ of $X$ after
$T_c(x,v,\tilde{x},\tilde{v})+T_{cc}(X_{T_c(x,v,\tilde{x},\tilde{v}})$,
and then $X_{\overline{T}(x,\tilde x)}=\tilde
X_{\overline{T}(x,\tilde x)}=0$ and $V_{\overline{T}(x,\tilde
x)}=\tilde V_{\overline{T}(x,\tilde x)}=1$ hold.
\end{prop}

\begin{proof}
From the previous sections, one can construct a coalescent
coupling with a coalescent time $T$  such that
$$
T(x,v,\tilde x,\tilde v)
\overset{\cL}{=}T_c(x,v,\tilde x,\tilde v)+T_{cc}(X_{T_c(x,v,\tilde x,\tilde v)}).
$$
Thanks to Lemmas~\ref{lem:somme}, \ref{lem:tc} and
\ref{lem:tcc}, we get that
\begin{align*}
T(x,v,\tilde x,\tilde v)&\leq_\mathrm{sto.} 
2E+\Sigma(2E)+S+S_{((x+\tilde x)/2,-1)}%
 +\frac{x-\tilde x}{2} + S_{((x-\tilde x)/2+E,+1)}+\Sigma(x-\tilde x)\\
&\leq_\mathrm{sto.} 2E+\Sigma(2E)+S_{(E,-1)}+S+\tilde S
+S_{(x,-1)}+\frac{x-\tilde x}{2}+\Sigma(x-\tilde x),
\end{align*}
 where $\tilde S$ is an independent copy of $S$ and all the
 random variables of the right hand side are independent.
 Recall that $E$ is equal to $F\wedge \tilde x$ where $F$
 is a random variable variable of law $\cE(a+b)$. In
 particular, $2E\leq_\mathrm{sto.}F+\tilde x$ and then
 $$
 T(x,v,\tilde x,\tilde v)\leq_\mathrm{sto.} \overline T(x,\tilde x),
 $$
where  $\overline T(x,\tilde x)$  is given by \eqref {eq:tb}. At last, 
for any $\lambda\leq \lambda_c$, one has
 $$
 \lambda+a(\psi(\lambda)-1)+c(\lambda)\leq 
 \lambda_c+a(\psi(\lambda_c)-1)+c(\lambda_c)=b-a< a+b.
 $$
 This ensures that for any $\lambda\leq \lambda_c$,
 \begin{align*}
 \dE\PAR{e^{\lambda F+\lambda\Sigma(2F)+ \lambda S_{(F,-1)}}}
 &= \dE\PAR{e^{(\lambda+a(\psi(\lambda)-1) +c(\lambda)) F}}
 =\frac{a+b}{2a+b-\lambda-a\psi(\lambda)-c(\lambda)}.
\end{align*}
Using the independence of the random variables provides
the desired upper bound.
\end{proof}

\begin{cor}
In particular, if $\tilde x\leq x$,
$$
\dE\PAR{e^{\lambda_c T(x,v,\tilde x,\tilde v)}}\leq 
\frac{(a+b)b}{2a^2} e^{r(a,b)x},
$$
where $\lambda_c$ is given in \eqref{eq:lc} and
$$
r(a,b)=\frac{3(b-a)}{4} \vee (b-\sqrt{ab}).
$$
\end{cor}
\begin{proof}
Let us choose $\lambda=\lambda_c$. Since
$$
\lambda_c=\frac{(\sqrt b-\sqrt a)^2}{2},\quad
c(\lambda_c)=\frac{b-a}{2}\quad\text{and}\quad
\psi(\lambda_c)=\sqrt{\frac{b}{a}},
$$
we get that
$$
x c(\lambda_c) +\frac{x+\tilde x}{2}\lambda_c
+\frac{x-\tilde x}{2}a(\psi(\lambda_c)-1)
=\frac{\sqrt b-\sqrt a}{4}\PAR{3x(\sqrt b+\sqrt a)+\tilde x(\sqrt b-3\sqrt a)}.
$$
If $x$ is fixed,  then the right hand side is a linear function of $\tilde x\in [0,x]$ 
and it is bounded above by the maximum of its values at $\tilde x\in\BRA{0,x}$. 
In other words,
$$
x c(\lambda) +\frac{x+\tilde x}{2}\lambda+\frac{x-\tilde x}{2}a(\psi(\lambda)-1)
\leq \frac{3}{4}(b-a) \vee (b-\sqrt{ab}).
$$
which concludes the proof.
\end{proof}

Using inequality  \eqref{coupleineg}, the  end of the proof of Theorem 
\ref{th:conv-reflection} is now obvious.

\section{The unreflected process}\label{sec:no-reflection}

We finally sketch the  proof of Theorem~\ref{th:no-reflection}. The invariant 
measure is  obtained by a similar regeneration argument as the one in 
Lemma~\ref{le:invmeas}, using  the obvious relation between 
excursions away from $(0,+1)$ of the reflected and unreflected processes, 
and Lemma~\ref{F}. The sketch of the proof of the 
bound~\eqref{eq:var-no-reflection} is the following:
\begin{itemize}
\item construct a coupling $(X,V,\tilde X,\tilde V)$ starting from
$(x, v,\tilde x,\tilde v)$ until time $\overline T(x,\tilde x)$, where
$$
x=\vert y\vert,
\quad
v=\mathrm{sgn}(y)w,
\quad
\tilde x=\vert \tilde y\vert,
\quad\text{and}\quad
\tilde v=\mathrm{sgn}(\tilde y)\tilde w,
$$
and notice that $X_{\overline T(x,\tilde x)}=\tilde X_{\overline T(x,\tilde x)}=0$
(see Proposition \ref{prop:tb});
\item construct $(Y,W,\tilde Y,\tilde W)$ on $[0,\overline T(x,\tilde x)]$
from $(X,V,\tilde X,\tilde W)$ and $(y,w,\tilde y,\tilde w)$ 
(see Section~\ref{sec:prelim}). Notice that 
$Y_{\overline T(x,\tilde x)}=\tilde Y_{\overline T(x,\tilde x)}=0$,
but in general $W_{\overline T(x,\tilde x)}=-\tilde W_{\overline T(x,\tilde x)}$;
\item wait for the first jump time  $E\sim\cE(2b)$   of  $(Y,W,\tilde Y,\tilde W)$  (as the minimum
of two independent random variables of law $\cE(b)$);
\item construct a coalescent coupling $(Y,W,\tilde Y,\tilde W)$
starting from $(E,w,-E,w)$ with a coupling time  smaller  than the hitting time of the origin
when starting at $(E,+1)$. 
\end{itemize}
We just give the details of the last point,  the other ones being clear. 
The construction is similar
to the one of $T_{cc}(x)$ for the reflected process. Assume that
$y=-\tilde y>0$ and $w=\tilde w=+1$ and consider two
independent random variables $(R,Q)$ with respective laws $\cE(a)$ and
$\cE(b)$. Then we may have:
\begin{itemize}
 \item Case 1: $R<y$. In this case, defining $T=R+Q$,
 $$
 \tilde W_t=
 \begin{cases}
 +1 & \text{if }t\in[0,R),\\
 -1 & \text{if }t\in[R,T),\\
 +1& \text{if }t=T,
  \end{cases}
 \quad\text{and}\quad
 W_t=
 \begin{cases}
 +1 & \text{if }t\in[0,Q),\\
 -1 & \text{if }t\in[Q,T), \\
 +1& \text{if }t=T,
 \end{cases}
$$
one has  $Y_T=y+Q-R=-\tilde Y_T$ and
$W_T=\tilde W_T=1$.
 \item Case 2:  $R\geq y$. In this case, defining $T=y+Q$,
 $$
 \tilde W_t=
 \begin{cases}
 +1 & \text{if }t\in[0,y),\\
 +1 & \text{if }t\in[y,T),\\
 -1& \text{if }t=T,
  \end{cases}
  \quad\text{and}\quad
 W_t=
 \begin{cases}
 +1 & \text{if }t\in[0,Q),\\
 -1 & \text{if }t\in[Q,T), 
 \end{cases}
$$
on has $Y_T=\tilde Y_T$ and $W_T=\tilde W_T=-1$.
At that time, $(Y,W)$ and $(\tilde Y,\tilde W)$ are coupled.
\end{itemize}
The algorithm to construct the paths $(Y,W,\tilde Y,\tilde W)$
consists in   repeating  the above construction until Case 2
occurs for the first time. This will happen before $Y$ reaches the
origin.
From this scheme and  previous work on the process $(X,V)$,
 the coupling time $S(y,w,\tilde y,\tilde w)$ satisfies
$$
S(y,w,\tilde y,\tilde w)\leq_{\mathrm{sto.}} \overline %
T(\ABS{y},\vert\tilde y\vert)+E+S_{(E,+1)}\overset{\cL}{=} %
\overline T(\ABS{y},\vert\tilde y\vert)+E+S+S_{(E,-1)}.
 $$
As a conclusion,
$
\dE\PAR{e^{\lambda S(y,w,\tilde y,\tilde w)}}\leq %
\dE\PAR{e^{\lambda \overline T(\ABS{y},\vert\tilde y\vert)}}
\psi(\lambda) \frac{2b}{2b-\lambda-c(\lambda)}.
$
In particular,
$$
\dE\PAR{e^{\lambda_c S(y,w,\tilde y,\tilde w)}}
\leq \PAR{\frac{b}{a}}^{5/2}\frac{a+b}{\sqrt{ab}+b}
 e^{r(a,b)x},
$$
where
$
r(a,b)=\frac{3(b-a)}{4} \vee (b-\sqrt{ab})
$. Using  \eqref{coupleineg} ends the proof.

\medskip

\noindent
{\bf Acknowledgements.} J. Fontbona thanks financial support from 
Fondecyt 1110923 and Basal-Conicyt, and the invitation and support of 
IRMAR (U. de  Rennes I).  F. Malrieu  acknowledges financial  support 
from ANR EVOL. All  three  authors thank an anonymous referee for 
suggestions that enabled  the improvement  of a former version of the paper.

\addcontentsline{toc}{section}{\refname}%
\bibliography{FontbonaGuerinMalrieuFinal}

\bigskip

\begin{flushright}\texttt{Compiled \today.}\end{flushright}

{\footnotesize %
  \noindent Joaquin \textsc{Fontbona},
 e-mail: \texttt{fontbona(AT)dim.uchile.cl}

 \medskip

 \noindent\textsc{CMM-DIM UMI 2807 UChile-CNRS, Universidad de Chile,
 Casilla 170-3, Correo 3, Santiago, Chile.}

 \bigskip

 \noindent H\'el\`ene \textsc{Gu\'erin},
e-mail: \texttt{helene.guerin(AT)univ-rennes1.fr}

 \medskip

  \noindent\textsc{UMR 6625 CNRS Institut de Recherche Math\'ematique de
    Rennes (IRMAR) \\ Universit\'e de Rennes 1, Campus de Beaulieu, F-35042
    Rennes \textsc{Cedex}, France.}

 \bigskip

 \noindent Florent \textsc{Malrieu},
 e-mail: \texttt{florent.malrieu(AT)univ-rennes1.fr}

 \medskip

  \noindent\textsc{UMR 6625 CNRS Institut de Recherche Math\'ematique de
    Rennes (IRMAR) \\ Universit\'e de Rennes 1, Campus de Beaulieu, F-35042
    Rennes \textsc{Cedex}, France.}

}

\end{document}